\documentclass[draft,11pt]{article}
\usepackage{geometry}
\usepackage{amsmath}
\usepackage{amssymb}
\usepackage{amsthm}
\usepackage{dsfont}					%to write indicator function
\usepackage{mathtools}				%to use \coloneqq to make symbol :=  look nice
\usepackage{color}
\usepackage{stackrel}				%put stuff below and on top of signs 			

\usepackage{centernot}				%to make a crossed out xleftrightarrow 
\usepackage{url}

\newtheorem{theorem}{Theorem}         
\numberwithin{theorem}{section}

\newtheorem{lemma}[theorem]{Lemma}
\newtheorem{proposition}[theorem]{Proposition}
\newtheorem{corollary}[theorem]{Corollary}

\theoremstyle{definition}    			%makes the newtheorem defined afterwards have plain font and not italic
\newtheorem{remark}[theorem]{Remark}
\newtheorem*{acknowledgement}{Acknowledgement}

\newcommand{\torus}{ {\mathbb T^d_N} }    %to not have to write always \mathbb T^d_N for the torus everywhere
\newcommand{\lattice}{ {\mathbb Z^d} }       %to not have to write always \mathbb Z^d for the lattice Z^d everywhere

\numberwithin{equation}{section}   %to number the equation following the section number and not just 1,2,3,4,5...

\begin{document}

\thispagestyle{empty}

\begin{center}
{\bf LOCAL PICTURE AND LEVEL-SET PERCOLATION OF THE GAUSSIAN FREE FIELD ON A LARGE DISCRETE TORUS} \\   
\vspace{0.4cm} 
{Angelo Ab\"acherli} \\

\end{center}

\vspace{1cm}
 
\begin{abstract}
In this article we obtain for $d\geq 3$ an approximation of the zero-average Gaussian free field on the discrete $d$-dimensional torus  of large side length $N$ by the Gaussian free field on $\lattice$, valid in boxes of roughly side length $N-N^\delta$ with $\delta \in (\frac12,1)$.

As an implication, the  level sets of the zero-average Gaussian free field on the torus can be approximated by the level sets of the Gaussian free field on $\lattice$. This leads to a series of applications related to level-set percolation. \end{abstract}

\vfill

\noindent Departement Mathematik   \\
ETH Z\"urich \\
R\"amistrasse 101 \\
8092 Z\"urich, Switzerland

\newpage
\thispagestyle{empty}
\mbox{}
\newpage

\setcounter{page}{1}

%SECTION 0 INTRODUCTION
\section{Introduction}  \label{section0}

We consider the zero-average Gaussian free field on the discrete $d$-dimensional torus of side length $N$ and fixed dimension $d\geq 3$. For large $N$, we show that  it can be approximated by the Gaussian free field on $\lattice$ in macroscopic boxes of side length of order $N-N^\delta$ for $\delta \in (\frac 12, 1)$,  thus yielding the local picture of the zero-average Gaussian free field. This readily provides an approximation of the level sets of the zero-average Gaussian free field on the torus by level sets of the Gaussian free field on $\lattice$, which in turn allows us to relate their respective percolative properties.

The general idea of tackling questions about large finite probabilistic models by comparing them to corresponding better understood infinite models has been fruitfully applied over the years in areas such as interacting particle systems (see e.g.~\cite{24}), combinatorial probability (see e.g.~\cite{25}) and spectral theory (see e.g.~\cite{26}).
Recently, this technique has been applied to the model of random interlacements on transient graphs \cite{14}, which can be used to approximate the trace of simple random walk on large finite graphs. In this way, the local picture of  the vacant set of simple random walk on finite graphs (see for example \cite{18}, \cite{15}, \cite{32}) and/or  percolative properties of it (see for example \cite{21}, \cite{5}, \cite{19}, \cite{6}) have been investigated. Some of our results are of similar flavour to \cite{5} and \cite{6} but in the Gaussian free field setting.

Level-set percolation for the Gaussian free field is a significant representative of a percolation model with long-range dependencies. It has attracted attention for a long time, dating back to \cite{31}, \cite{28} and \cite{12}. More recent developments can be found for instance in \cite{29}, \cite{7}, \cite{23}, \cite{13} and \cite{22}. Also, some simulations of the critical value of level-set percolation were performed, see \cite{30}.

We now describe our results in more details. The graphs considered in this work are the discrete tori $\torus \coloneqq (\mathbb Z / N \mathbb Z)^d$, $N\geq 1$, and the discrete lattice $\lattice$, both for dimensions $d\geq 3$ and endowed with the usual nearest-neighbour structure. On $\torus$ we consider the zero-average Gaussian free field (see Subsection \ref{subsection1.3} for more details about it) with law $\mathbb P^\torus$ on $\mathbb R^\torus$ and canonical coordinate process $(\Psi_\torus(x))_{x \in \torus}$ so that,
\begin{equation} \label{0.1}
\begin{split}
&\text{under $\mathbb P^\torus$, $(\Psi_\torus(x))_{x \in \torus}$ is a  centered Gaussian field on $\torus$} \\
&\text{with covariance $\mathbb E^\torus [\Psi_\torus(x) \Psi_\torus(y)] = G_\torus(x,y)$ for all $x,y \in \torus$}
\end{split}
\end{equation}
where $G_\torus(\cdot,\cdot)$ is the zero-average  Green function, see \eqref{1.5}. On the other hand, on $\lattice$ we have the Gaussian free field with law $\mathbb P^\lattice$ on $\mathbb R^\lattice$ and canonical coordinate process $(\varphi_\lattice(x))_{x \in \lattice}$ so that,
\begin{equation} \label{0.2}
\begin{split}
&\text{under $\mathbb P^\lattice$, $(\varphi_\lattice(x))_{x \in \lattice}$ is a  centered Gaussian field on $\lattice$} \\
&\text{with covariance $\mathbb E^\lattice [\varphi_\lattice(x) \varphi_\lattice(y)] = g_\lattice(x,y)$ for all $x,y \in \lattice$}
\end{split}
\end{equation}
where $g_\lattice(\cdot,\cdot)$ stands for the Green function of simple random walk on $\lattice$, see \eqref{1.1} and again Subsection \ref{subsection1.3} for more details.

Our main result is the following: for $\delta \in (\frac 12,1)$ we let $\mathcal B_N^\delta$ denote a box of roughly side length $N-N^\delta$ in $\lattice$ centered at the origin. Moreover, we let $B_N^\delta$ be the box in $\torus$ obtained from $\mathcal B_N^\delta$ via the canonical projection $\pi_N: \lattice \to \torus$. For $x \in B_N^\delta \subseteq \torus$ we denote  the unique element in $\pi_N^{-1}(\{ x\}) \cap \mathcal B_N^\delta$ by $\widehat x \in \mathcal B_N^\delta \subseteq \lattice$. We will prove in Theorem \ref{theorem2.2} that
\begin{equation} \label{0.3}
\begin{split}
&\text{there exist couplings $\mathbb Q_N$, $N\geq1$, of the zero-average Gaussian  free field $\Psi_\torus$} \\
&\text{on $\torus$ and the Gaussian free field $\varphi_\lattice$ on $\lattice$ such that for every $\delta \in \textstyle (\frac12,1)$} \\
&\text{and $\varepsilon>0$, $\lim_{N \to \infty} \mathbb Q_N \big[ \textstyle \sup_{x \in  B_N^\delta} \big| \Psi_\torus(x) - \varphi_\lattice(\widehat x)  \big|  >  \varepsilon \big] =0 $.}
\end{split}
\end{equation}
We refer to Section \ref{section2}, Theorem \ref{theorem2.2}, for the precise (and more quantitative) statement. 
A direct consequence of $\eqref{0.3}$ is an approximation of the level sets $E_{\Psi_\torus}^{\geq h} \coloneqq \{x\in \torus \, | \,  \Psi_\torus(x)\geq h\}$ of $\Psi_\torus$ restricted to $B_N^\delta$ by level sets $E_{\varphi_\lattice}^{\geq h} \coloneqq \{x\in \lattice \, | \, \varphi_\lattice(x)\geq h\}$ of $\varphi_\lattice$ restricted to $\mathcal B_N^\delta$. In essence, we show in Corollary \ref{corollary2.3} that
\begin{equation} \label{0.4}
\begin{split}
&\text{the couplings $\mathbb Q_N$, $N\geq 1$, in \eqref{0.3} satisfy that for every $\delta \in \textstyle (\frac12,1)$, $\varepsilon >0$ and} \\
&\text{$h \in \mathbb R$, $\lim_{N \to \infty} \mathbb Q_N \Big[ \pi_N \big( E_{\varphi_\lattice}^{\geq h+\varepsilon} \cap \mathcal B_N^\delta  \big)  \subseteq  \big( E_{\Psi_\torus}^{\geq h} \cap  B_N^\delta \big) \subseteq \pi_N \big( E_{\varphi_\lattice}^{\geq h-\varepsilon} \cap \mathcal B_N^\delta \big)   \Big]=1$.} 
\end{split}
\end{equation}
The result \eqref{0.4} provides an analogue (for level sets of the Gaussian free fields $\Psi_\torus$ and $\varphi_\lattice$) of the approximation of the vacant set of simple random walk on $\torus$ by the vacant set of random interlacements on $\lattice$ obtained first in \cite{5}, Theorem 1.1, for boxes of side length $N^{1-\delta}$ with $\delta \in (0,1)$, and later improved to boxes of side length $(1-\delta)N$  with $\delta \in (0,1)$  in \cite{6}, Theorem 1.2. Note that in our setting the approximation of the level sets of $\Psi_\torus$ by level sets of $\varphi_\lattice$ holds for boxes of even larger size $(1-N^{\delta-1})N$ with $\delta \in (\frac12,1)$.

Incidentally, let us mention that one cannot expect an approximation of $\Psi_\torus$ by $\varphi_\lattice$ as in $\eqref{0.3}$ that goes `up to the boundary of the torus' (i.e.~in boxes of side length $N-\delta$ with $\delta \geq 0$) due to the differences in the global structure of $\torus$ and $\lattice$ emerging at this scale, see Remark \ref{extraremark1}.

The approximation \eqref{0.4} leads to a series of applications related to level-set percolation of the Gaussian free field (similar to \cite{5} and \cite{6} for random interlacements). As a reminder, the critical value of level-set percolation of $\varphi_\lattice$ can be defined by (see \cite{7}, equation (0.4))
\begin{equation} \label{0.5}
h_\star \coloneqq \inf \Big \{  h\in \mathbb R \, \Big| \,   \mathbb P^\lattice \big[ 0 \xleftrightarrow{\varphi_\lattice \geq h} \infty \big] = 0 \Big \},
\end{equation}
where $\{ 0 \xleftrightarrow{\varphi_\lattice \geq h} \infty \}$ denotes the event of the existence of an infinite connected component of the level set $E_{\varphi_\lattice}^{\geq h}$ containing $0 \in \lattice$. Moreover, for values $h\in \mathbb R$ above a second critical parameter $h_{\star\star}$ (defined in \cite{7}, equation (0.6)), the connectivity function $\mathbb P^\lattice [ 0 \xleftrightarrow{\varphi_\lattice \geq h} x ]$ of the $h$-level set (i.e.~the probability of 0 and $x \in \lattice$ being in the same connected component of $E_{\varphi_\lattice}^{\geq h}$) decays fast in $|x|$ (see \cite{7}, Theorem 2.6, later improved to an exponential decay when $d\geq 4$, with a logarithmic correction when $d=3$, in \cite{13}, Theorem 2.1). It is known that $0\leq h_\star \leq h_{\star\star} < \infty$ for $d\geq 3$ (see \cite{12}, Theorem 3, and \cite{7}, Theorem 2.6 and Corollary 2.7). Additionally, $h_\star >0$ for $d$ large enough (see \cite{7}, Theorem 3.3, and \cite{23}, Theorem 0.2 and Theorem 0.3). An open question remains if $h_\star >0$ for all $d\geq 3$ and whether actually $h_\star = h_{\star\star}$ (cf.~\cite{23}). For what concerns level-set percolation, we will show in Theorem \ref{theorem3.1} that
\begin{equation} \label{0.6}
\begin{split}
&\text{in the subcritical phase $h>h_\star$,  with high probability for large $N$, the} \\
&\text{level set $E_{\Psi_\torus}^{\geq h}$ of $\Psi_\torus$ contains no macroscopic connected component} \\
&\text{(i.e.~of size comparable to the volume of the torus).}
\end{split}
\end{equation}
Furthermore, Theorem \ref{theorem3.2} states that								
\begin{equation} \label{0.7}
\begin{split}
&\text{for $h>h_{\star\star}$, with high probability for large $N$, the connected compo-} \\
&\text{nents of the level set $E_{\Psi_\torus}^{\geq h}$ are all microscopic  (i.e.~of negligible size} \\
&\text{compared to to the volume of the torus).}
\end{split}
\end{equation}
On the other hand, Theorem \ref{theorem3.4} proves that
\begin{equation} \label{0.8}
\begin{split}
&\text{in the supercritical phase $h < h_\star$, with high probability for large $N$,} \\
&\text{the level set $E_{\Psi_\torus}^{\geq h}$ of $\Psi_\torus$ contains a macroscopic connected compo-} \\
&\text{nent in diameter sense.}
\end{split}
\end{equation}
Let us briefly comment on the proof of the main Theorem \ref{theorem2.2} (corresponding to \eqref{0.3}). The starting point is a coupling of $\Psi_\torus$ and $\varphi_\lattice$ based on conditional distributions of the two fields (Lemma \ref{lemma1.2}). It then remains to bound the variances of the averages of $\Psi_\torus$ and $\varphi_\lattice$ appearing in this coupling (Proposition \ref{proposition2.1}), which relies on the long-range decay of the covariance functions of $\Psi_\torus$ and $\varphi_\lattice$, something well known in the case of $\varphi_\lattice$ (see \eqref{1.2}). For $\Psi_\torus$ this key ingredient is established in Proposition \ref{proposition1.1}. 

The structure of the article is as follows. In Section \ref{section1} we introduce the notation, recall basic facts on the Green functions and Gaussian free fields, and prove the long-range decay of the covariance function of $\Psi_\torus$. Then in Section \ref{section2} we deduce the main approximation result (Theorem \ref{theorem2.2}, corresponding to \eqref{0.3}) and its corollary on level sets (Corollary \ref{corollary2.3}, corresponding to \eqref{0.4}). We conclude in Section \ref{section3} with the applications to level-set percolation (Theorems \ref{theorem3.1}, \ref{theorem3.2} and \ref{theorem3.4}, corresponding to \eqref{0.6}, \eqref{0.7} and \eqref{0.8}).

A final word on the convention followed concerning constants. By $c,c',\ldots$, we denote  positive constants with values changing from place to place which only depend on the dimension $d$. 
%Numbered constants $c_0,c_1,\ldots$ are defined in the place of first occurrence and thereafter remain fixed. The dependence of constants on additional parameters is going to appear in the notation. 

%SECTION 1 NOTATION AND USEFUL RESULTS
\section{Notation and useful results}  \label{section1}

We start this section with further notation. In Subsection \ref{subsection1.1} resp.~Subsection \ref{subsection1.2} we then collect basic properties and definitions related to random walks and Green functions on $\lattice$ resp.~$\torus$. In particular, we obtain in Subsection \ref{subsection1.2} some bound on the long-range decay of the zero-average Green function (which plays the role of the covariance function of $\Psi_\torus$), see Proposition \ref{proposition1.1}. This will be particularly useful to prove the main result Theorem \ref{theorem2.2} (corresponding to $\eqref{0.3}$) in Section \ref{section2}. Finally, in Subsection \ref{subsection1.3} we provide a concise expression relating the laws of $\varphi_\lattice$ and $\Psi_\torus$, see Lemma \ref{lemma1.2}, the starting point for the construction of the couplings in \eqref{0.3} carried out in Section \ref{section2}.

As mentioned above, we consider $\torus = (\mathbb Z / N \mathbb Z)^d$, $N\geq 1$, and $\lattice$ for $d\geq 3$ endowed with the usual graph structure. We let $d_\torus(\cdot,\cdot)$ resp.~$d_\lattice(\cdot,\cdot)$ denote the graph distance on $\torus$ resp.~on $\lattice$ and write $x\sim_\torus \! y$ resp.~$x\sim_\lattice \! y$ for neighbouring vertices in $\torus$ resp.~in $\lattice$. Moreover, for $x\in \lattice$ we let $|x |$ stand for the Euclidean norm on $\lattice$. For any set $U\subseteq \torus$ resp.~$U \subseteq \lattice$ we denote the (outer) boundary of $U$ in $\torus$ resp.~in $\lattice$ by $\partial_\torus U \coloneqq \{ y \in \torus \setminus U \, | \, y \sim_\torus \!\! x  \text{ for some } x\in U \}$ resp.~$\partial_\lattice U \! \coloneqq \! \{ y \in \lattice \setminus U \, | \,  y \sim_\lattice x \text{ for some } x\in U \}$. Furthermore, $|U|$ stands for the cardinality of $U$. We let $\pi_N: \lattice \to \torus$ denote the canonical projection. If $x\in \torus$, we write $\widehat x \in \lattice$ for the unique element of $\pi_N^{-1}(\{ x \}) \cap (-\frac{N}2,\frac{N}2]^d$. If $U \subseteq \torus$, we similarly write $\widehat U \coloneqq \{ \widehat x \in \lattice \, | \, x \in U \} \subseteq \lattice$. Note that the map $x \mapsto \widehat x$  from  $\torus$ to $(-\frac{N}2,\frac{N}2]^d \cap \lattice$ is a bijection with inverse $\pi_N |_{(-\frac{N}2,\frac{N}2]^d \cap \lattice}$. 

We now introduce some notation concerning simple random walk on $\torus$ and $\lattice$. We write $P_x^\torus$ resp.~$P_x^\lattice$ for the canonical law of the simple random walk on $\torus$ resp.~on $\lattice$ starting at $x$ as well as $E_x^\torus$ resp.~$E_x^\lattice$ for the corresponding expectation. The canonical process for both discrete-time walks is denoted by $(X_k)_{k\geq 0}$. For the continuous-time walks with i.i.d. holding times of distribution Exp(1) we write $(\overline{X}_t)_{t\geq 0}$. Given a vertex set $U\subseteq \torus$ resp.~$U\subseteq \lattice$, we write $T_U \coloneqq \inf \{ k\geq 0 \, | \, X_k \notin U\}$ for the exit time from $U$ and $H_U \coloneqq \inf \{ k\geq 0 \, | \, X_k \in U \}$ for the entrance time in $U$ of the discrete-time walk $(X_k)_{k\geq 0}$ on $\torus$ resp.~on $\lattice$. In the special case of $U=\{z \}$ we use $H_z$ in place of $H_{\{z\}}$.

%SUBSECTION 1.1 SIMPLE RANDOM WALKS AND GREEN FUNCTIONS	
\subsection{Simple random walk and Green functions on $\lattice$} \label{subsection1.1}

The Green function $g_\lattice(\cdot,\cdot)$ of simple random walk on $\lattice$ is
\begin{equation}  \label{1.1}
g_\lattice(x,y) \coloneqq E_x^\lattice \Big[\sum_{k=0}^\infty \mathds{1}_{\{X_k=y\}} \Big] = \sum_{k=0}^\infty P_x^\lattice[X_k=y] \quad \text{for } x,y \in \lattice.
\end{equation}
It is symmetric, positive, finite and satisfies $g_\lattice(x,y)=g_\lattice(x-y,0)$. Furthermore (see \cite{3}, Theorem 1.5.4),
\begin{equation}  \label{1.2}
g_\lattice(x,y) \sim c |x-y|^{2-d}, \quad \text{as } |x-y| \longrightarrow \infty. 
\end{equation}
For $U\subseteq \lattice$, the Green function $g_\lattice^U(\cdot,\cdot)$ of simple random walk on $\lattice$ killed when exiting $U$ is
\begin{equation} \label{1.3}
g_\lattice^U(x,y) \coloneqq E_x^\lattice \Big[\sum_{0\leq k < T_U} \mathds{1}_{\{X_k=y\}} \Big] = \sum_{k=0}^\infty P_x^\lattice[X_k=y, k<T_U] \quad \text{for } x,y \in \lattice.
\end{equation}
It is again symmetric, finite and vanishes whenever $x\notin U$ or $y\notin U$. The functions $g_\lattice(\cdot,\cdot)$ and $g_\lattice^U(\cdot,\cdot)$ are related by the identity (see \cite{2}, Proposition 4.6.2(a))
\begin{equation} \label{1.4}
g_\lattice(x,y) = g_\lattice^U(x,y) + E_x^\lattice \big[g_\lattice(X_{T_U},y) \mathds{1}_{\{T_U < \infty\}}\big] \quad \text{for } x,y \in \lattice,
\end{equation}
which follows from a simple application of the strong Markov property of $(X_k)_{k\geq0}$.

Next we show a basic property of simple random walk on $\lattice$, namely an easy estimate of the hitting distribution of a hyperplane. It is needed for the proof of the bounds in Proposition \ref{proposition2.1}. 

\begin{lemma} \label{extralemma1}
Consider the half-space $\mathbb H \coloneqq \{ (y_1,\ldots,y_d) \in \lattice \, | \, y_1\geq 1 \}$ of $\lattice$ with boundary $\partial \mathbb H \coloneqq \partial_\lattice \mathbb H = \{ (y_1,\ldots,y_d) \in \lattice \, | \, y_1=0 \}$. Then for all $x\in \mathbb H$ and $z \in \partial \mathbb H$ it holds that
\begin{equation}  \label{extraequation20}
P_x^\lattice[ X_{T_{\mathbb H}} = z] \leq c|x-z|^{1-d}. 
\end{equation} 
\end{lemma}

\begin{proof}
The proof uses the method of images. Fix $x\in \mathbb H$ and $z \in \partial \mathbb H$. If $z'$ denotes the unique vertex in $\mathbb H$ such that $z' \sim_\lattice z$, then by \cite{2}, Lemma 6.3.6, one has (note that  $P_x^\lattice$-almost surely $T_{\mathbb H} < \infty$)
\begin{equation}  \label{extraequation15}
P_x^\lattice [ X_{T_{\mathbb H}} = z ]  = \frac{1}{2d} g_\lattice^{\mathbb H}(x,z')  
\end{equation}
(where we use the notation \eqref{1.3}).
By \eqref{1.4} we have 
\begin{equation} \label{extraequation16}
g_\lattice^{\mathbb H}(x,z') = g_\lattice(x,z') - E_x^\lattice \big[g_\lattice(X_{T_{\mathbb H}},z') \big].
\end{equation} 
If we let $\overline{x} \in \lattice$ be the vertex obtained by reflection of $x=(x_1,\ldots,x_d)$ at $\partial \mathbb H$ (i.e.~$\overline{x}=(-x_1,x_2,\ldots,x_d) \notin \mathbb H$), then we obtain, by symmetry and the strong Markov property,
\begin{equation}   \label{extraequation17}
\begin{split} 
E_x^\lattice \big[&g_\lattice(X_{T_{\mathbb H}},z')  \big] = E_{\overline x}^\lattice \big[g_\lattice(X_{H_{\partial \mathbb H}},z') \big] = E_{\overline x}^\lattice \Big[\sum_{k  \geq H_{\partial \mathbb H}} \!\!\! \mathds{1}_{\{X_k=z'\}} \Big] \overset{\overline{x}\notin \mathbb H,z' \in \mathbb H }{=}  g_\lattice(\overline x,z').
\end{split}
\end{equation}
Again by symmetry, the expression on the right hand side of \eqref{extraequation17} equals  $g_\lattice(x,\overline z')$ with $\overline z'$ denoting the vertex obtained by reflection of $z'$ at $\partial \mathbb H$. Thus, combining \eqref{extraequation15}, \eqref{extraequation16} and \eqref{extraequation17} we deduce
\begin{equation} \label{extraequation18}
P_x^\lattice [ X_{T_{\mathbb H}} = z]  = \frac1{2d} ( g_\lattice(x,z') - g_\lattice(x,\overline z') ). 
\end{equation}
On the right hand side of \eqref{extraequation18} we have a discrete gradient of the Green function $g_\lattice(\cdot,\cdot)$, which can be bounded by \cite{3}, Theorem 1.5.5, by
\begin{equation} \label{extraequation19}
\begin{split}
g_\lattice(x,z') - g_\lattice(x,\overline z')  \leq c|x-z|^{1-d} 
\end{split}
\end{equation}
(note that $|z'-\overline z'|=2$). The statement \eqref{extraequation20} follows from \eqref{extraequation18} and \eqref{extraequation19}. 
\end{proof}

%SUBSECTION 1.2 SIMPLE RANDOM WALKS AND GREEN FUNCTION ON TORUS
\subsection{Simple random walk and Green functions on $\torus$} \label{subsection1.2}

The zero-average Green function $G_\torus(\cdot,\cdot)$ associated with simple random walk on $\torus$ is given by
\begin{equation} \label{1.5}
G_\torus(x,y) \coloneqq  \int_0^\infty \Big( P_x^\torus[\overline{X}_t = y] - \frac1{N^d} \Big) \, dt \quad \text{for } x,y \in \torus.
\end{equation}
It is symmetric, satisfies $G_\torus(x,y) = G_\torus(x-y,0)$, is finite and positive-semidefinite, i.e.~for any $f: \torus \to \mathbb R$ one has $\sum_{x,y \in \torus} f(x) G_\torus(x,y) f(y) \geq 0$. 
We explain in the next remark how  these last two properties can be derived. 
\begin{remark}  
If $Q=(Q_{xy})_{x,y \in \torus}$ denotes the $N^d\times N^d$ transition rate matrix (generator matrix) of the continuous-time simple random walk $(\overline{X}_t)_{t\geq 0}$ on $\torus$,  that is, 
\begin{equation*}
Q_{xx} \coloneqq -1 \text{ for } x \in \torus \quad \text{and} \quad  Q_{xy} \coloneqq P_x^\torus[X_1 = y] \text{ for } x,y \in \torus, x\neq y,
\end{equation*}
then by the spectral representation (see \cite{1}, equation (3.33) in Chapter 3, p.~73) one has\begin{equation} \label{100}
P_x^\torus[\overline{X}_t = y] =  \sum_{k=1}^{N^d} u^k(x)  \exp(-\mu_k t) u^k(y),
\end{equation} 
where $\mu_1,\ldots,\mu_{N^d} \in \mathbb R$ are the eigenvalues of $-Q$ and $u^1,\ldots,u^{N_d} \in \mathbb R^{\torus}$ are  corresponding orthonormal eigenvectors.
The eigenvalues of $-Q$ are given in  \cite{1}, Example 5.17 in Chapter 5, p.~190. In particular, we can order them in such a way that 
\begin{equation} \label{101}
0=\mu_1 < \mu_2 \leq \ldots \leq \mu_{N^d} \quad  \text{and} \quad  u^1=\frac1{\sqrt{N^d}} (1,\ldots,1)^T 
\end{equation}
without loss of generality since the eigenvectors are normalised. The above formula \eqref{100} implies the finiteness of $G_\torus(x,y)$ because
\begin{equation} \label{102}
\begin{split}
G_\torus(x,y) &\overset{\eqref{100}}{=}  \int_0^\infty \Big( u^1(x) \exp(-\mu_1 t) u^1(y) + \sum_{k=2}^{N^d}  u^k(x)  \exp(-\mu_k t) u^k(y) - \frac1{N^d} \Big) \, dt \\
&\overset{\eqref{101}}{=} \sum_{k=2}^{N^d} u^k(x)u^k(y) \int_0^\infty \exp(-\mu_k t) \, dt  =  \sum_{k=2}^{N^d} u^k(x)u^k(y)  \frac1{\mu_k} 
\end{split}
\end{equation}
(see also  \cite{1}, equation (3.41) in Chapter 3, p.~74) and hence $\big|G_\torus(x,y) \big| < \infty$ for all $x,y\in \torus$ by \eqref{101}. Moreover, it also implies that for any $f: \torus \to \mathbb R$ one has
\begin{equation*}
\sum_{x,y \in \torus} f(x) G_\torus(x,y) f(y) \overset{\eqref{102}}{=}     \sum_{k=2}^{N^d} \frac1{\mu_k}  \Big( \sum_{z \in \torus} f(z) u^k(z) \Big)^2 \; \overset{\eqref{101}}{\geq} 0, 
\end{equation*}
showing the positive-semidefiniteness of $G_\torus(\cdot,\cdot)$. \qed
\end{remark}

\begin{remark} \label{extraremark23}
Consider $x,y \in \torus$ and define \textup{o}$\ \coloneqq \pi_N(0) \in \torus$ for all $N\geq 1$.  As we now explain one has $\big| G_\torus(x,y) \big|  \leq  G_\torus(\textup{o},\textup{o}) \xrightarrow{N \to \infty} g_\lattice(0,0)$.
Indeed for the inequality, the semi-definiteness of $G_\torus(\cdot,\cdot)$  applied to $f: \torus \to \mathbb R$ with $f(x) = G_\torus(y,y)^\frac12$, $f(y) = -G_\torus(x,y) / G_\torus(y,y)^\frac12$ and $f(z)=0$ for all $z \neq x,y$ implies $G_\torus(x,x)G_\torus(y,y)-G_\torus(x,y)^2 \geq 0$, that is $\big| G_\torus(x,y) \big| \leq G_\torus(x,x)^{\frac12}  G_\torus(y,y)^{\frac12} \overset{\eqref{1.5}}{=} G_\torus(\textup{o},\textup{o})$.

For the convergence of $G_\torus(\textup{o},\textup{o})$ to $g_\lattice(0,0)$ as $N$ goes to infinity, note that by  \cite{1}, Proposition 13.8 in Chapter 13, p.~428, one has $\lim_{N\to \infty}  G_\torus(\textup{o},\textup{o}) = R^d$ (combine therein (13.37) with the first equality of (13.39)), where  $R^d$ satisfies that $R^d = \lim_{m \to \infty}  E_0^\lattice \Big[\sum_{k=0}^m \mathds{1}_{\{X_k=0\}} \Big] \overset{\eqref{1.1}}{=} g_\lattice(0,0)$ by (13.31) of the same reference.
%By the Cauchy-Schwarz inequality and \cite{1}, Proposition 13.8 in Chapter 13, p.~428, one has that $G_\torus(x,y)  \leq G_\torus(x,x)^{\frac12}  G_\torus(y,y)^{\frac12} = G_\torus(\textup{o},\textup{o}) \xrightarrow{N \to \infty} g_\lattice(0,0)$.  
\qed  
\end{remark}
For later use we define for $U\subsetneq \torus$ the Green function of simple random walk on $\torus$ killed when exiting $U$, which is
\begin{equation} \label{1.6}
g_\torus^U(x,y) \coloneqq E_x^\torus \Big[\sum_{0\leq k < T_U} \mathds{1}_{\{X_k=y\}} \Big] = \sum_{k=0}^\infty P_x^\torus[X_k=y, k<T_U] \quad \text{for } x,y \in \torus
\end{equation}
(equally, $g_\torus^U(\cdot,\cdot)$ could be defined through the continuous-time simple random walk).
As $g_\lattice^U(\cdot,\cdot)$, it is symmetric, finite and vanishes for $x\notin U$ or $y\notin U$.
The functions $G_\torus(\cdot,\cdot)$ and $g_\torus^U(\cdot,\cdot)$ are related by a similar expression as the identity \eqref{1.4} for the Green functions on $\lattice$, namely
\begin{lemma}  
Assume $U \subsetneq \torus$. Then it holds that
\begin{equation}  \label{1.7}
G_\torus(x,y) = g_\torus^U(x,y) + E_x^\torus \big[G_\torus(X_{T_U},y) \big] -\frac{1}{N^d} E_x^\torus [T_U]  \quad \text{for all } x,y \in \torus.
\end{equation}
\end{lemma}

\begin{proof}
We first prove the statement for the case  $U \coloneqq \torus \setminus \{z\}$ with $z \in \torus$ fixed. Note that in this case $T_U = H_z$ and hence $g_\torus^U(\cdot,\cdot)$ is the Green function of simple random walk on $\torus$ killed when hitting $z$. If we assume $x\neq z$, then \cite{1}, Lemma 2.9 in Chapter 2, p.~29, considering Subsection 2.2.3, p.~34, reads as 
\begin{equation}  \label{extraequation1}
g_\torus^U(x,y) = \frac1{N^d} \Big(E_x^\torus[H_z] + E_z^\torus[H_y] - E_x^\torus[H_y] \Big). 
\end{equation} 
Now apply \cite{1}, Lemma 2.12 in Chapter 2, p.~29, considering Subsection 2.2.3, p.~34, to the second and third expectation in \eqref{extraequation1} (observe that $Z_{xy}$ in the notation of \cite{1} corresponds to $G_\torus(x,y)$) and obtain (note that it also trivially holds for $x=z$) 
\begin{equation}  \label{extraequation2}
g_\torus^{\torus \setminus \{z\}}(x,y) = \frac1{N^d} E_x^\torus[H_z] - G_\torus(z,y) + G_\torus(x,y) \quad \text{for all } x,y \in \torus.
\end{equation}
Rearranging \eqref{extraequation2} leads to \eqref{1.7} for the special case $U=\torus \setminus \{z \}$ since $P_x^\torus$-almost surely $X_{T_U} = z$. Now assume that $U \subsetneq \torus$ is arbitrary. Choose $z \in \torus\setminus U$ and assume $x\in U$. Then, $P_x^\torus$-almost surely $T_U \leq H_z<\infty$. We compute in the same way as when proving $\eqref{1.4}$  by the strong Markov property 
\begin{equation} \label{extraequation3}
\begin{split}
g_\torus^{\torus \setminus \{z\}} (x,y) &\overset{\eqref{1.6}}{=} E_x^\torus \Big[\sum_{0\leq k < T_U} \mathds{1}_{\{X_k=y\}} \Big]  + E_x^\torus \Big[\sum_{T_U\leq k < H_z} \mathds{1}_{\{X_k=y\}} \Big] \\
&\overset{\phantom{\eqref{1.6}}}{=} g_\torus^U (x,y) + E_x^\torus \big [g_\torus^{\torus \setminus \{z\}}(X_{T_U},y) \big].
\end{split}
\end{equation}
We can use \eqref{extraequation2} inside the expectation on the right-hand side of  \eqref{extraequation3} and obtain for all $x\in U$ and $y\in \torus$
\begin{equation} \label{extraequation4}
\begin{split}
&g_\torus^{\torus \setminus \{z\}} (x,y)  = g_\torus^U(x,y) +  E_x^\torus \Big[ \frac1{N^d} E_{X_{T_U}}^\torus[H_z] -G_\torus(z,y)+ G_\torus(X_{T_U},y)  \Big] \\
&\qquad \qquad \quad \, = g_\torus^U(x,y) + \frac1{N^d}E^\torus_x[H_z - T_U] - G_\torus(z,y) + E_x^\torus \big[  G_\torus(X_{T_U},y)  \big] 
\end{split}
\end{equation}
again by the strong Markov property. By comparing the right hand sides of \eqref{extraequation4} and \eqref{extraequation2} we deduce \eqref{1.7} for all $x\in U$ and $y\in \torus$. Since \eqref{1.7} trivially holds if $x\notin U$, the proof of \eqref{1.7} is complete.
\end{proof}

We conclude this subsection proving an upper bound of $G_\torus(x,y)$ for $x,y \in \torus$ `far apart', a key ingredient for later (see Proposition \ref{proposition2.1}). Recall that for example on $\lattice$ one has \eqref{1.2}.

\begin{proposition} \label{proposition1.1}
For all $N\geq 1$ and $x,y \in \torus$ with  $x\neq y$ it holds that
\begin{equation} \label{1.8}
\big|G_\torus(x,y) \big| \leq c(\ln(N))^{\frac{3d}{2}} d_\torus(x,y)^{2-d}.
\end{equation}
\end{proposition}

\begin{proof}
The strategy is  to split the integral appearing in the definition of $G_\torus(x,y)$  at the time the walk $(\overline{X}_t)_{t\geq0}$ reaches equilibrium (here the uniform distribution). The time to reach equilibrium is roughly the inverse of the spectral gap of $(X_k)_{k\geq 0}$ (cf.~\cite{4}, Theorem 20.6). We now formalise this procedure.

The spectral gap of  simple random walk $(X_k)_{k\geq 0}$ on $\torus$ is (combine in  \cite{4} Subsection 12.3.1 with Corollary 12.12)
\begin{equation} \label{1.10}
\lambda_\torus \coloneqq \frac1d (1-\cos(\frac{2\pi}{N})) \sim \frac{2\pi^2}{dN^2}, \quad  \text{as } N \longrightarrow \infty.
\end{equation}
We set 
\begin{equation} \label{1.11}
t_\torus \coloneqq \frac1{\lambda_\torus} \ln(N^d) \sim \frac{d^2}{2\pi^2} N^2 \ln(N), \quad \text{as } N \longrightarrow \infty.
\end{equation}
By \cite{4}, Theorem 20.6, we now have
\begin{equation} \label{1.12}
\int_{t_\torus}^\infty \Big| P_x^\torus[\overline{X}_t = y] - \frac1{N^d}  \Big| \, dt \leq \int_{t_\torus}^\infty  e^{-\lambda_\torus t} \, dt \overset{\eqref{1.11}}{=}\frac1{\lambda_\torus N^d} \overset{\eqref{1.10}}{\leq} cN^{2-d} \leq c d_\torus(x,y)^{2-d}.
\end{equation}
On the other hand, switching to the discrete walk $(X_k)_{k\geq 0}$ (with $N_t$, $t \geq 0$, the number of jumps of the continuous-time simple random walk up to time $t$), we have
\begin{equation} \label{1.13}
\begin{split}
&\int_0^{t_\torus} \Big| P_x^\torus[\overline{X}_t = y] - \frac1{N^d}  \Big| \, dt  \leq \int_0^{t_\torus} \sum_{k=0}^\infty P_x^\torus[N_t = k] P_x^\torus[X_k = y] \, dt + \frac{t_\torus}{N^d}  \\
&\leq \sum_{k=0}^{\lfloor e t_\torus \rfloor} P_x^\torus[X_k = y]  \underbrace{\int_0^\infty \frac{t^k}{k!}e^{-t} \, dt}_{=1} + \int_0^{t_\torus} P_x^\torus[N_t \geq e t_\torus] \, dt + \frac{t_\torus}{N^d}  \\
&\overset{(*)}{\leq}  \sum_{k=0}^{\lfloor e t_\torus \rfloor} P_x^\torus[X_k = y]  + t_\torus \exp(-t_\torus) + \frac{t_\torus}{N^d},
\end{split}
\end{equation}
where in $(*)$  we used the exponential Markov inequality.  Because of \eqref{1.11}, the second and third term on the right hand side of \eqref{1.13} are clearly bounded by $c \ln(N) N^{2-d} \leq c \ln(N) d_\torus(x,y)^{2-d}$. Therefore, it only remains to bound the sum appearing in the last line of \eqref{1.13}, which can be rewritten as 
\begin{equation} \label{1.14}
\begin{split}
\sum_{k=0}^{\lfloor e t_\torus \rfloor} P_x^\torus[X_k=y] = \sum_{k=0}^{\lfloor e t_\torus \rfloor} \sum_{v \in \lattice} P_{\widehat x}^\lattice[X_k = \widehat y+vN].
\end{split}
\end{equation}
The inner series on the right hand side of \eqref{1.14} is actually a finite sum because  $P_{\widehat x}^\lattice[X_k = \widehat y+vN]=0$ for  $v\in \lattice$ with $d_\lattice(\widehat x, \widehat y+vN) > \lfloor e t_\torus \rfloor$ since $k\leq \lfloor e t_\torus  \rfloor$. Additionally, for $v\in \lattice$ with $(\lfloor e t_\torus \rfloor \geq)$ $d_\lattice(\widehat x,\widehat y+vN) \geq \sqrt{e t_\torus} (\ln(N))$ we can apply \cite{2}, Proposition 2.1.2(b), and find
\begin{equation*}
P_{\widehat x}^\lattice[X_k =\widehat y+vN] \leq P_{\widehat x}^\lattice \Big[\max_{k=0,\ldots,\lfloor e t_\torus  \rfloor}  d_\lattice(X_k, \hat x)   \geq \sqrt{e t_\torus} \ln(N) \Big] \leq c e^{-c' (\ln(N))^{2}}.
\end{equation*} 
There are at most $c(\frac1N t_\torus)^d$ different such $v\in \lattice$ because $\lfloor e t_\torus \rfloor \geq  d_\lattice(\widehat x,\widehat y+vN)$.
Finally, for the remaining $v\in \lattice$ with  $d_\lattice(\widehat x,\widehat y+vN) < \sqrt{e t_\torus} \ln(N)$ we can exchange the two sums and for each such $v$  get 
\begin{equation*} 
\begin{split}
\sum_{k=0}^{\lfloor e t_\torus \rfloor}  P_{\widehat x}^\lattice[X_k = \widehat y+vN] \overset{\eqref{1.1}}{\leq} g_\lattice(\widehat x,\widehat y+vN) \overset{\eqref{1.2}}{\leq} c d_\lattice(\widehat x,\widehat y+vN)^{2-d} \leq c d_\torus(x,y)^{2-d}.
\end{split}
\end{equation*}
There are at most $c(\frac1N \sqrt{t_\torus} \ln(N))^d$ different such $v\in \lattice$ because $d_\lattice(\widehat x,\widehat y+vN) < \sqrt{e t_\torus} \ln(N)$. All together,  by \eqref{1.14} we obtain
\begin{equation} \label{extraequation22} 
\sum_{k=0}^{\lfloor e t_\torus \rfloor}  P_x^\torus[X_k =y]  \leq c t_\torus \Big(\frac1N t_\torus \Big)^d e^{-c' (\ln(N))^{2}} + c''\Big(\frac1N \sqrt{t_\torus} \ln(N) \Big)^d d_\torus(x,y)^{2-d}.
\end{equation}
Because of \eqref{1.11}, the expression on the right hand side of \eqref{extraequation22} is bounded by the right hand side of \eqref{1.8} since the first term is bounded by $cN^{2-d} \leq c d_\torus(x,y)^{2-d}$ and the second term is bounded by $c (\ln(N))^{\frac{3d}{2}} d_\torus(x,y)^{2-d}$. Hence,
the combination of \eqref{1.12} and \eqref{1.13} concludes the proof.
\end{proof}

%SUBSECTION 1.2 THE GAUSSIAN FREE FIELDS PSI AND PHI	
\subsection{The Gaussian free fields $\Psi_\torus$ and $\varphi_\lattice$} \label{subsection1.3}

We now turn to the Gaussian free fields $\Psi_\torus$ and $\varphi_\lattice$. Recall from \eqref{0.1} and \eqref{1.5} that $(\Psi_\torus(x))_{x \in \torus}$ is the centered Gaussian field with the zero-average Green function $G_\torus(\cdot,\cdot)$ as covariance, whereas recall from \eqref{0.2} and \eqref{1.1} that $(\varphi_\lattice(x))_{x \in \lattice}$ is the centered Gaussian field with the Green function $g_\lattice(\cdot,\cdot)$ of simple random walk on $\lattice$ as covariance. The next remark concerns the denotation `zero-average Gaussian free field' for $\Psi_\torus$ resp.~`zero-average Green function' for $G_\torus(\cdot,\cdot)$.  

\begin{remark}  \label{extraremark6}
Note that by bilinearity 
$$\textup{Var}_{\mathbb P_\torus} \Big(\sum_{x\in \torus} \Psi_\torus(x) \Big) = \sum_{x,y \in \torus} G_\torus(x,y) \overset{\eqref{1.5}}{=} 0,$$
which shows the zero-average property of $\Psi_\torus$ and $G_\torus(\cdot,\cdot)$. \qed
\end{remark} 

As we now explain, for any fixed subset $U \subsetneq \torus$ the zero-average Gaussian free field $\Psi_\torus$ can be related to the Gaussian free field on the torus vanishing outside $U$ (i.e.~with covariance $g_\torus^U(\cdot,\cdot))$. This will be of importance for the coupling of $\Psi_\torus$ and $\varphi_\lattice$ in Lemma \ref{lemma1.2}.

\begin{lemma} \label{extralemma2}
Consider $U \subsetneq \torus$. For $x \in \torus$ set 
\begin{equation}  \label{extraequation5}
\varphi_\torus^U(x) \coloneqq  \Psi_\torus(x) - E_x^\torus[ \Psi_\torus(X_{T_U}) ].
\end{equation} 
Then (recall \eqref{1.6}),
\begin{equation} \label{1.19}
\begin{split}
&\text{under $\mathbb P^\torus$, $(\varphi_\torus^U(x))_{x\in \torus}$ is a centered Gaussian field on $\torus$} \\
&\text{with covariance $\mathbb E^\torus [\varphi_\torus^U(x) \varphi_\torus^U(y)] = g_\torus^U(x,y)$ for all $x,y \in \torus$.}
\end{split}
\end{equation}
\end{lemma}

\begin{proof}
The fact that $\varphi_\torus^U$ is centered is clear from its definition \eqref{extraequation5}. Now consider $x,y \in \torus$ and expand the covariance to obtain
\begin{equation} \label{extraequation6}
\begin{split}
\mathbb E^\torus [\varphi_\torus^U(x) \varphi_\torus^U(y)] &\overset{\eqref{extraequation5}}{=} G_\torus(x,y) - E_x^\torus[ G_\torus(X_{T_U},y)]  - E_y^\torus[ G_\torus(x,X_{T_U})]  \\ 
& \qquad \qquad  +  \mathbb E^\torus \Big[E_x^\torus[ \Psi_\torus(X_{T_U}) ]  E_y^\torus[ \Psi_\torus(X_{T_U})] \Big]. 
\end{split}
\end{equation} 
We introduce the product measure $\overline{P}_{x,y}^\torus \coloneqq P_x^\torus \times P_y^\torus$  so that $((X_k)_{k\geq0},(\widetilde{X}_k)_{k\geq0})$ under $\overline{P}_{x,y}^\torus$ has independent components distributed as a simple random walk on $\torus$ started at $x$ and a simple random walk on $\torus$ started at $y$. We let  $\overline{E}_{x,y}^\torus$ be the corresponding expectation.  Then, the last term on the right hand side of \eqref{extraequation6} equals
\begin{equation}  \label{extraequation7}
\mathbb E^\torus \Big[\overline{E}_{x,y}^\torus \big[ \Psi_\torus(X_{T_U})  \Psi_\torus(\widetilde X_{T_U}) \big] \Big]  \overset{\eqref{0.1}}{=} \overline{E}_{x,y}^\torus \big[ G_\torus(X_{T_U},\widetilde{X}_{T_U}) \big ].
\end{equation}
Additionally, the first two and the third term on the right hand side of \eqref{extraequation6} satisfy
\begin{equation} \label{extraequation8}
\begin{split}
G_\torus(x,y) - E_x^\torus[ G_\torus(X_{T_U},y)]  &\overset{\eqref{1.7}}{=} g_\torus^U(x,y) - \frac1{N^d} E_x^\torus [T_U]   \\ 
E_y^\torus[ G_\torus(x,X_{T_U})]  &\overset{\eqref{1.7}}{=} \overline{E}_{x,y}^\torus \big[ G_\torus(X_{T_U},\widetilde{X}_{T_U}) \big ] -  \frac{1}{N^d} E_x^\torus [T_U] . 
\end{split}
\end{equation} 
The combination of  \eqref{extraequation6}, \eqref{extraequation7} and \eqref{extraequation8} concludes the proof.
\end{proof}

\begin{remark} 
1) Fix any $z\in \torus$. If one applies the above Lemma \ref{extralemma2} for the choice $U \coloneqq  \torus \setminus \{ z\}$,  one obtains that $\big(\varphi_\torus^{\torus \setminus \{ z\}}(x)\big)_{x\in \torus} = (\Psi_\torus(x) - \Psi_\torus(z))_{x\in \torus}$ has the law of the Gaussian free field on $\torus$ pinned down at $z$ (i.e.~with covariance $g_\torus^{\torus \setminus \{ z\}}(\cdot,\cdot)$). Incidentally, the zero-average Gaussian free field $\Psi_\torus$ can be seen as the Gaussian free field on $\torus$ pinned down at $z$ and shifted by the random constant $\Psi_\torus(z)$. Since by Remark \ref{extraremark6} we have $\mathbb P^\torus$-almost surely the identity $\Psi_\torus(x) = \frac1{N^d} \sum_{z\in \torus}(\Psi_\torus(x) - \Psi_\torus(z))$, this implies that the field $\Psi_\torus$ can alternatively also be seen as an average of Gaussian free fields on $\torus$ pinned down at the various locations $z \in \torus$. 

\noindent 2) Recall that a property similar to Lemma \ref{extralemma2} holds for the Gaussian free field $\varphi_\lattice$ on $\lattice$. Indeed, for $U \subseteq \lattice$ and $(\varphi_\lattice^U(x))_{x\in \lattice}$ with
\begin{equation*}
\varphi_\lattice^U(x) \coloneqq  \varphi_\lattice(x) - E_x^\lattice[ \varphi_\lattice(X_{T_U}) \mathds{1}_{\{T_U < \infty\}} ] \quad \text{for } x\in \lattice,
\end{equation*}
one has that (recall \eqref{1.3}),
\begin{equation} \label{1.18}
\begin{split}
&\text{under $\mathbb P^\lattice$, $(\varphi_\lattice^U(x))_{x\in \lattice}$ is a centered Gaussian field on $\lattice$} \\
&\text{with covariance $\mathbb E^\lattice [\varphi_\lattice^U(x) \varphi_\lattice^U(y)] = g_\lattice^U(x,y)$ for all $x,y \in \lattice$.}
\end{split}
\end{equation}
This essentially follows from expanding the covariance and using \eqref{1.4} in a similar way as we used \eqref{1.7} in the proof of \eqref{1.19}. The relation \eqref{1.18} is part of the domain Markov property of $\varphi_\lattice$ (see e.g.~\cite{7}, Lemma 1.2).  \qed
\end{remark}

We conclude Section \ref{section1} by coupling $\Psi_\torus$ and $\varphi_\lattice$ in a straightforward way keeping \eqref{1.19} and \eqref{1.18} in mind.
First, we introduce a geometric condition on subsets of the torus central for this coupling and also needed  for later on. From the beginning of Section \ref{section1}, we remind the notation $\widehat x$ for the unique vertex in the pre-image of $x\in \torus$ under the projection $\pi_N:\lattice \to \torus$ lying in $(-\frac{N}2,\frac{N}2]^d$  and similarly $\widehat U \subseteq (-\frac{N}2,\frac{N}2]^d\cap \lattice$ for the set $\{ \widehat x \in \lattice \, | \, x \in U \}$ for $U \subseteq \torus$. Note that for $U \subseteq \torus$ we do not necessarily have $\partial_\lattice \widehat U = \widehat{\partial_\torus U}$ or equivalently that $\partial_\lattice \widehat U$ is contained in $(-\frac{N}2,\frac{N}2]^d\cap \lattice$.
We say that $U \subseteq \torus$ is \emph{properly contained} in $\torus$ if 
\begin{equation} \label{importantcondition}
\begin{split}
\text{$\partial_\lattice \widehat{U} \subseteq \textstyle (-\frac{N}2,\frac{N}2 ]^d \cap \lattice$.} 
\end{split}
\end{equation}
\begin{remark} \label{extraremark4}
Condition \eqref{importantcondition} on $U \subseteq \torus$ implies that $U \neq \torus$. More importantly, it ensures that for any $x\in \torus$ the image under $\pi_N$ of the law of the simple random walk on $\lattice$ started at $\widehat x \in \lattice$ and stopped when exiting $\widehat U$ is the same as the  law of  the simple random walk on $\torus$ started at $x$ and stopped when exiting $U$. In particular, the hitting distribution of the boundary $\partial_\torus U$ of the walk on $\torus$ is the image under $\pi_N$ of the hitting distribution of $\partial_\lattice \widehat{U}$ of the walk on $\lattice$. More precisely, it holds that
\begin{equation} \label{extraequation21}
P_x^\torus[X_{T_U} = y]   =  P_{\widehat x}^\lattice[X_{T_{\widehat  U}} = \widehat y] \quad \text{for all } x \in U \text{ and } y \in \partial_\torus U.
\end{equation} \qed
\end{remark}

\begin{lemma} \label{lemma1.2}
Assume $U \subseteq \torus$ is properly contained in $\torus$, that is, $U$ satisfies \eqref{importantcondition}. Then there exists a coupling of $\Psi_\torus$ and $\varphi_\lattice$ satisfying
\begin{equation} \label{1.20}
\Psi_\torus(x) - E_x^\torus[ \Psi_\torus(X_{T_{U}}) ]  = \varphi_\lattice(\widehat{x}) - E_{\widehat{x}}^\lattice[ \varphi_\lattice(X_{T_{\widehat U}}) ] \quad  \text{for all } x \in \torus.
\end{equation} 
\end{lemma}

\begin{proof}
By \eqref{1.19} the left hand side of \eqref{1.20} describes a centered Gaussian field with covariance $(g_\torus^U(x,y))_{x,y\in \torus}$, whereas by \eqref{1.18} the right hand  side describes a centered Gaussian field with covariance $(g_\lattice^{\widehat U}(\widehat{x},\widehat{y}))_{x,y\in \torus}$. Since from  \eqref{importantcondition}, see also Remark \ref{extraremark4}, we know $g_\torus^U(x,y) = g_\lattice^{\widehat U}(\widehat{x},\widehat{y})$ for all $x,y \in \torus$, the proof is complete. 
\end{proof}

%SECTION 2: PROOF OF THE MAIN RESULTS
\section{Proof of the main results}  \label{section2}

The main goal of this section is the proof of the precise version of \eqref{0.3}, in the form of Theorem \ref{theorem2.2} below, stating the approximation  in macroscopic boxes of the zero-average Gaussian free field $\Psi_\torus$ on $\torus$ by the Gaussian free field $\varphi_\lattice$ on $\lattice$. The strategy for proving it is to combine the coupling of $\Psi_\torus$ and $\varphi_\lattice$ from Lemma \ref{lemma1.2}, for some suitable choice of $U \subsetneq \torus$, with uniform bounds in $x$ of the variances of the expectations appearing in \eqref{1.20}. These uniform bounds are shown in Proposition \ref{proposition2.1}. After the proof of Theorem \ref{theorem2.2} we conclude the section with the derivation of the precise version of \eqref{0.4} in the form of Corollary \ref{corollary2.3}. It concerns the approximation of level sets of $\Psi_\torus$ by level sets of $\varphi_\lattice$ and directly follows from Theorem \ref{theorem2.2}. 

Recall from  Section \ref{section1}  the bijective map $x \mapsto \widehat x$  from  $\torus$ to $(-\frac{N}2,\frac{N}2]^d \cap \lattice$. Observe that $(-\frac{N}2,\frac{N}2]^d \cap \lattice = \{-\lceil \frac{N}{2}\rceil +1 ,\ldots, \lfloor \frac{N}{2} \rfloor  \}^d$.
We introduce the boxes 
\begin{align}
\mathcal U_N &\coloneqq \begin{cases} \{0\}^d \subseteq \lattice, &\text{if } N=1,2  \\ \{-\lceil \frac{N}{2}\rceil +2 ,\ldots, \lfloor \frac{N}{2} \rfloor -1  \}^d \subseteq \lattice, &\text{if } N \geq 3  \end{cases}   \label{extraequation13} \\
U_N &\coloneqq \pi_N(\mathcal U_N) \subsetneq \torus \quad \text{for } N \geq 1.    \label{extraequation11}
\end{align}
Additionally, for $\delta \in (\frac12,1)$ we also define the following boxes (possibly empty for $N$ small)  of side length roughly $\lfloor N-N^\delta \rfloor$
\begin{align}
\mathcal B_N^\delta &\coloneqq \Big(-\frac{N-N^\delta-2}2,\frac{N-N^\delta-2}2 \Big]^d \cap \lattice \subseteq \lattice  \label{extraequation14} \\
B_N^\delta &\coloneqq \pi_N(\mathcal B_N^\delta) \subseteq \torus.   \label{extraequation12}
\end{align}
%Note that $\widehat{U_N} = \mathcal U_N$ and $\widehat{B_N^\delta} = \mathcal B_N^\delta$. Furthermore, $B_N^\delta \subseteq U_N$ and $\mathcal B_N^\delta \subseteq \mathcal U_N$. Finally, $U_N$ is by definition properly contained in $\torus$ since condition \eqref{importantcondition} is met.
%One has 
%\begin{align*} 
%&\widehat{U_N} = \mathcal U_N \text{ and } \widehat{B_N^\delta} = \mathcal B_N^\delta  \\
%&B_N^\delta \subseteq U_N \text{ and } \mathcal B_N^\delta \subseteq \mathcal U_N \\
%&U_N \text{ is  properly contained in $\torus$ since condition \eqref{importantcondition} is met}.
%\end{align*}
Note that $\widehat{U_N} = \mathcal U_N$ and $\widehat{B_N^\delta} = \mathcal B_N^\delta$. Furthermore, $B_N^\delta \subseteq U_N$ and $\mathcal B_N^\delta \subseteq \mathcal U_N$. Finally, 
\begin{equation} \label{extraequation24} 
U_N \text{ is properly contained in $\torus$ since condition \eqref{importantcondition} is met}.
\end{equation}

\begin{proposition} \label{proposition2.1}
For all $\delta \in (\frac12,1)$, $N\geq 1$ and $x \in \mathcal B_N^\delta$ one has
\begin{equation} \label{2.1}  
\textup{Var}_{\mathbb P^\lattice}  \Big(  E_x^\lattice[\varphi_\lattice(X_{T_{\mathcal U_N}})]  \Big)  \leq cN^{-(2\delta-1)\frac{(d-2)(d-1)}{2d-3}}.   \vspace{-1pt}
\end{equation}
Moreover, for all $\delta \in (\frac12,1)$, $N\geq 1$ and $x \in B_N^\delta$ one has
\begin{equation}   \label{2.2}
\textup{Var}_{\mathbb P^\torus}  \Big(  E_x^\torus[\Psi_\torus(X_{T_{U_N}})]  \Big) \leq c (\ln(N))^{\frac{3d}2} N^{-(2\delta-1)\frac{(d-2)(d-1)}{2d-3}}.      \vspace{-1pt}
\end{equation}
\end{proposition}

\begin{remark}
Actually, for the proof of Theorem \ref{theorem2.2} we only need, for each $\delta \in (\frac12,1)$, bounds of the form $c_\delta N^{-c'_\delta}$  of the variances   in \eqref{2.1} and \eqref{2.2} for some $c_\delta,c'_\delta > 0$ depending on $\delta$ (and $d$). \qed
\end{remark} 

\begin{proof}
We first prove \eqref{2.1}. Consider $x \in \mathcal B_N^\delta$. Expanding the variance in \eqref{2.1} we obtain for $\gamma \in (0,1)$ to be chosen later (below \eqref{2.4}) that
\begin{equation} \label{2.3}
\begin{split}
\textup{Var}_{\mathbb P^\lattice} \Big( E_x^\lattice[\varphi_\lattice(X_{T_{\mathcal U_N}})]\Big) = &\sum_{\substack{y,z \in \partial_\lattice \mathcal U_N \\ d_\lattice(y,z)<N^\gamma}} \!\!\!\!\! P_x^\lattice[X_{T_{\mathcal U_N}}=y] P_x^\lattice[X_{T_{\mathcal U_N}}=z] g_\lattice(y,z) \\
&+ \!\!\!\!\! \sum_{\substack{y,z \in \partial_\lattice \mathcal U_N \\ d_\lattice(y,z)\geq N^\gamma}}  \!\!\!\!\! P_x^\lattice[X_{T_{\mathcal U_N}}=y] P_x^\lattice[X_{T_{\mathcal U_N}}=z] g_\lattice(y,z).
\end{split}
\end{equation}
In the first sum on the right-hand side of \eqref{2.3} we use the bound $g_\lattice(y,z) \leq g_\lattice(0,0) = c$ for all $y,z$ (see \eqref{1.1}). Moreover, $P_x^\lattice[X_{T_{\mathcal U_N}}=y]$ is bounded by $c d_\lattice(x,y)^{1-d}$ by Lemma \ref{extralemma1}. (Indeed, choose $i\in \{1,\ldots,N^d\}$ and $\sigma \in \{\pm 1\}$  such that the half-space $\mathbb H^\sigma_i \coloneqq \{(u_1,\ldots,u_d) \in \lattice \, | \, \sigma u_i \geq 1 \}$ shifted by $y$ contains $\mathcal U_N$, i.e.~$\mathcal U_N \subseteq \mathbb  H^\sigma_i + y \coloneqq \{ u+y \, | \, u \in \mathbb  H^\sigma_i \}$. Then one has $\{ X_{T_{\mathcal U_N}}=y \} \subseteq \{ T_{\mathcal U_N} = T_{\mathbb H^\sigma_i +y} \}$ and so $P_x^\lattice[X_{T_{\mathcal U_N}}=y] = P_x^\lattice[X_{T_{\mathcal U_N}}=y, T_{\mathcal U_N} = T_{\mathbb  H^\sigma_i+y} ] \leq P_x^\lattice[X_{T_{\mathbb  H^\sigma_i +y }}=y] = P_{x-y}^\lattice[X_{T_{\mathbb  H^\sigma_i}}=0]$ using translation invariance in the last step. The latter quantity is bounded by $c|x-y|^{1-d} \leq c d_\lattice(x,y)^{1-d}$ by Lemma \ref{extralemma1} and symmetry.) 
Since $x \in \mathcal B_N^\delta$ and $y \in \partial_\lattice \mathcal U_N$, we find that $P_x^\lattice[X_{T_{\mathcal U_N}}=y] \leq c(N^\delta)^{1-d}$ by \eqref{extraequation13} and \eqref{extraequation14}. The same bound can be applied to $P_x^\lattice[X_{T_{\mathcal U_N}}=z]$. Note that the number of pairs $y,z$ in this first sum is at most $cN^{d-1} (N^\gamma)^{d-1}$.

In the second sum on the right hand side of \eqref{2.3} we use  \eqref{1.2} to estimate $g_\lattice(y,z) \leq c d_\lattice(y,z)^{2-d}  \leq cN^{\gamma(2-d)}$ for all $y,z$. What remains of the sum is bounded   by one. Thus, we have obtained that
\begin{equation} \label{2.4}
\textup{Var}_{\mathbb P^\lattice} \Big( E_x^\lattice[\varphi_\lattice(X_{T_{\mathcal U_N}})]\Big)  \leq cN^{2\delta(1-d)}N^{(1+\gamma)(d-1)} + c'N^{\gamma(2-d)}. \vspace{-1.5pt}
\end{equation}
Choosing $\gamma = \frac{(2\delta-1)(d-1)}{2d-3}$  the two expressions on the right hand side of \eqref{2.4} are of the same order and \eqref{2.1} follows (note that $\gamma \in (0,1)$ since $\delta \in (\frac12,1)$). 

For the proof of \eqref{2.2} we proceed as for \eqref{2.1}. Consider $x$ in $B_N^\delta$. Expanding the variance in \eqref{2.2} and using \eqref{extraequation21} one obtains for $\gamma \in (0,1)$ to be chosen later (below \eqref{2.6}) that
\begin{equation} \label{2.5}
\begin{split}
\textup{Var}_{\mathbb P^\torus} \Big( E_x^\torus[\Psi_\torus(X_{T_{U_N}})]\Big) =  \!\!\!\!&\sum_{\substack{y,z \in \partial_\torus U_N \\ d_\torus(y,z)< N^\gamma}} \!\!\!\!\! P_{\widehat x}^\lattice[X_{T_{\mathcal U_N}}= \widehat y] P_{\widehat x}^\lattice[X_{T_{\mathcal U_N}}= \widehat z] G_\torus(y,z) \\
&+  \!\!\!\!\! \sum_{\substack{y,z \in \partial_\torus  U_N \\ d_\torus(y,z)\geq N^\gamma}} \!\!\!\!\! P_{\widehat x}^\lattice[X_{T_{\mathcal U_N}}=\widehat y] P_{\widehat x}^\lattice[X_{T_{\mathcal U_N}}=\widehat z] G_\torus(y,z).
\end{split}
\end{equation}
We bound the first sum in \eqref{2.5} as in \eqref{2.3}. Again $P_{\widehat x}^\lattice[X_{T_{\mathcal U_N}}= \widehat y] \leq c(N^\delta)^{1-d}$ holds by Lemma \ref{extralemma1} and the same holds for $\widehat z$. On the other hand, we have $G_\torus(y,z) \leq c$ by Remark \ref{extraremark23}. In the second sum in \eqref{2.5} we use \eqref{1.8} to estimate $G_\torus(y,z)$ for all $y,z$. In this way, we obtain from \eqref{2.5}  that
\begin{equation} \label{2.6}
\textup{Var}_{\mathbb P^\torus} \Big( E_x^\torus[\varphi_\torus(X_{T_{U_N}})]\Big)  \leq cN^{2\delta(1-d)}N^{(1+\gamma)(d-1)} + c'(\ln(N))^{\frac{3d}2} N^{\gamma(2-d)}. \vspace{-1.5pt}
\end{equation}
Again choosing $\gamma = \frac{(2\delta-1)(d-1)}{2d-3}$ the bound \eqref{2.2} follows. 
\end{proof}

The next theorem is the main result of this section.

\begin{theorem} \label{theorem2.2}
For any $N\geq 1$ there exists a coupling $\mathbb Q_N$ of $\Psi_\torus$ and $\varphi_\lattice$ such that for all $\delta \in (\frac12,1)$ and $\varepsilon >0$ there is $c_\delta,c'_\delta > 0$ depending on $\delta$ (and $d$) with
\begin{equation} \label{2.7}
\mathbb Q_N \Big[ \sup_{x \in B_N^\delta} \big| \Psi_\torus(x) - \varphi_\lattice(\widehat x)  \big|  >  \varepsilon \Big] \leq  4 N^d \exp(-c_\delta\varepsilon^2 N^{c'_\delta})
\end{equation} 
(where $B_N^\delta \subseteq \torus$ is the box in \eqref{extraequation12}). In particular,  for all $\delta \in (\frac12,1)$ and $\varepsilon >0$  one has
\begin{equation} \label{2.8}
\lim_{N \to \infty} \mathbb Q_N \Big[ \sup_{x \in B_N^\delta} \big| \Psi_\torus(x) - \varphi_\lattice(\widehat x)  \big|  >  \varepsilon \Big] =0.
\end{equation} 
\end{theorem}

\begin{proof}
Consider  the box $U_N \subsetneq \torus$ in \eqref{extraequation11}. By \eqref{extraequation24} the assumption in  Lemma \ref{lemma1.2} is satisfied and we obtain a coupling $\mathbb Q_N$  of $\Psi_\torus$ and $\varphi_\lattice$ such that for all $\delta \in (\frac12,1)$ and $\varepsilon>0$
\begin{equation} \label{2.9}
\begin{split}
\mathbb Q_N \Big[ &\sup_{x \in B_N^\delta} \big| \Psi_\torus(x) - \varphi_\lattice(\widehat x)   \big|  >  \varepsilon \Big] \\
&\leq \sum_{x \in B_N^\delta} \Bigg( \mathbb Q_N \bigg[ \Big| E_{\widehat x}^\lattice[ \varphi_\lattice(X_{T_{\mathcal U_N}}) ]\Big| > \frac{\varepsilon}2 \bigg]  +   \mathbb Q_N \bigg[ \Big| E_x^\torus[ \Psi_\torus(X_{T_{U_N}})] \Big| > \frac{\varepsilon}2 \bigg]  \Bigg).
\end{split}
\end{equation} 
The expectations appearing in \eqref{2.9} are centered Gaussian variables with respect to $\mathbb P^\lattice$ and $\mathbb P^\torus$. Thus, for any $x\in B_N^\delta$, the exponential Markov inequality  leads to
\begin{equation} \label{2.10}
\begin{split}
&\mathbb Q_N \bigg[ \Big| E_{\widehat x}^\lattice[ \varphi_\lattice(X_{T_{\mathcal U_N}}) ]\Big| > \frac{\varepsilon}2 \bigg] \\
&\qquad \qquad \qquad\leq 2 \exp\bigg(-\frac{(\varepsilon/2)^2}{2 \textup{Var}_{\mathbb P^\lattice} \big( E_{\widehat x}^\lattice[\varphi_\lattice(X_{T_{\mathcal U_N}})]\big)} \bigg)  
 \overset{\eqref{2.1}}{\leq} 2e^{-c\varepsilon^2 N^{c_\delta}}  \\
&\mathbb Q_N \bigg[ \Big| E_x^\torus[ \Psi_\torus(X_{T_{U_N}}) ]\Big| > \frac{\varepsilon}2 \bigg] \\
&\qquad \qquad \qquad\leq 2 \exp\bigg(-\frac{(\varepsilon/2)^2}{2 \textup{Var}_{\mathbb P^\torus} \big( E_x^\torus[\Psi_\torus(X_{T_{U_N}})]\big)} \bigg) 
 \overset{\eqref{2.2}}{\leq} 2e^{-c_\delta \varepsilon^2 N^{c'_\delta}}
\end{split}
\end{equation}
for some $c_\delta, c'_\delta > 0$. The desired result follows from the combination of  \eqref{2.10} and \eqref{2.9} (note that $|B_N^\delta| \leq N^d$). 
\end{proof}

\begin{remark}   \label{extraremark1}
1) Considering the precise exponents in \eqref{2.1} and \eqref{2.2} when bounding in \eqref{2.10} shows that  one can actually take any
$c'_\delta < (2\delta-1) \frac{(d-2)(d-1)}{2d-3}$ in \eqref{2.7}.  

\noindent 2) Let us mention that the asymptotic approximation in \eqref{2.8} becomes false if we replace $B_N^\delta$ by a box of side length $N-\delta$ with $\delta\geq 0$  centered at zero since the different structure of $\torus$ and $\lattice$ starts to play a role. 
In that case one can choose vertices $x_N$ and $y_N$, $N\geq1$, at `opposite borders' of the box which remain at fixed distance in $\torus$ as $N \to \infty$ but for which the $\lattice$-distance between $\widehat{x_N}$ and $\widehat{y_N}$ tends to infinity. So the covariance $g_\lattice(\widehat{x_N},\widehat{y_N})$ between $\varphi_\lattice(\widehat{x_N})$ and $\varphi_\lattice(\widehat{y_N})$ tends to zero (see \eqref{1.2}) but the covariance $G_\torus(x_N, y_N)$ between $\Psi_\torus(x_N)$ and $\Psi_\torus(y_N)$ stays bounded away from zero (actually, one can show that $G_\torus(x_N,y_N) \to  g_\lattice(\widehat{x}_1,\widehat{y_1})>0$ as $N\to \infty$). From this fact one readily infers that the corresponding statement to \eqref{2.8} breaks down. \qed
\end{remark}

The main Theorem \ref{theorem2.2} provides a tool to compare the level sets of $\Psi_\torus$ and $\varphi_\lattice$. Recall from the introduction the notation 
\begin{equation}  \label{2.11}
E_{\Psi_\torus}^{\geq h} = \{x\in \torus \, | \,  \Psi_\torus(x)\geq h\} \quad \text{and} \quad  E_{\varphi_\lattice}^{\geq h} = \{x\in \lattice \, | \, \varphi_\lattice(x)\geq h\} 
\end{equation}
for the level sets of the Gaussian free fields, where $h\in \mathbb R$.

\begin{corollary}  \label{corollary2.3}
For any $N\geq 1$ there exists a coupling $\mathbb Q_N$ of $\Psi_\torus$ and $\varphi_\lattice$ such that for all $\delta \in (\frac12,1)$,  $\varepsilon >0$ and $h \in \mathbb R$ there is $c_\delta,c'_\delta > 0$ depending on $\delta$ (and $d$) with
\begin{equation}  \label{2.12}
\mathbb Q_N \Big[ \pi_N \big( E_{\varphi_\lattice}^{\geq h+\varepsilon} \cap \mathcal B_N^\delta  \big)  \subseteq  \big( E_{\Psi_\torus}^{\geq h} \cap B_N^\delta \big) \subseteq \pi_N \big( E_{\varphi_\lattice}^{\geq h-\varepsilon} \cap \mathcal B_N^\delta \big)   \Big]  \geq 1 -4 N^d \exp(-c_\delta\varepsilon^2 N^{c'_\delta})
\end{equation}
(where $B_N^\delta \subseteq \torus$ and $\mathcal B_N^\delta \subseteq \lattice$ are the boxes in \eqref{extraequation12} and \eqref{extraequation14}). In particular, for all $\delta \in (\frac12,1)$, $\varepsilon >0$ and $h \in \mathbb R$ one has 
\begin{equation}  \label{2.13}
\lim_{N\to \infty} \mathbb Q_N \Big[ \pi_N \big( E_{\varphi_\lattice}^{\geq h+\varepsilon} \cap \mathcal B_N^\delta  \big)  \subseteq  \big( E_{\Psi_\torus}^{\geq h} \cap B_N^\delta \big) \subseteq \pi_N \big( E_{\varphi_\lattice}^{\geq h-\varepsilon} \cap \mathcal B_N^\delta \big)   \Big] =1.
\end{equation}
\end{corollary}

\begin{proof}
The event $\{ \sup_{x \in B_N^\delta} | \Psi_\torus(x) - \varphi_\lattice(\widehat x)  |  \leq  \varepsilon  \}$ is contained in the event inside the probability in \eqref{2.12}. Therefore, the statement follows from Theorem \ref{theorem2.2}.
\end{proof}
\begin{remark}
By Remark \ref{extraremark1} any $c'_\delta < (2\delta-1) \frac{(d-2)(d-1)}{2d-3}$ can be chosen as constant in \eqref{2.12}. \qed
\end{remark}

%SECTION 3: APPLICATIONS TO LEVEL-SET PERCOLATION
\section{Applications to level-set percolation}  \label{section3}

As mentioned in the introduction Corollary \ref{corollary2.3} is a (stronger) analogue of \cite{5}, Theorem 1.1, and \cite{6}, Theorem 1.2, in the Gaussian free field setting. As such it has  implications concerning level-set percolation of the Gaussian free field. Recall the two critical parameters $h_\star$ and $h_{\star\star}$ from the introduction (see \eqref{0.5} and thereafter) and also the notation $E_{\Psi_\torus}^{\geq h}$ resp.~$E_{\varphi_\lattice}^{\geq h}$ with $h\in \mathbb R$ from \eqref{2.11}. 
We further denote by $\mathcal C_{\text{max}}^{\torus,h}$ an arbitrary connected component of $E_{\Psi_\torus}^{\geq h}$ with maximal number of vertices (we will only be interested in its cardinality) and by $\mathcal C_{x}^{\torus,h}$ resp.~$\mathcal C_{x}^{\lattice,h}$ the connected component of $E_{\Psi_\torus}^{\geq h}$ resp.~$E_{\varphi_\lattice}^{\geq h}$ containing $x \in \torus$ resp.~$x \in \lattice$. Finally, we abbreviate o$\ = \pi_N(0) \in \torus$ for all $N\geq 1$.
%{\color{red} We further denote  
%\begin{align*}
%&\text{by } \mathcal C_{\text{max}}^{\torus,h} \text{ an arbitrary connected component of }  E_{\Psi_\torus}^{\geq h} \text{with maximal number of vertices},   \\
%&\text{by } \mathcal C_{x}^{\torus,h} \text{ the connected component of } E_{\Psi_\torus}^{\geq h} \text{ containing } x \in \torus,  \\
%&\text{by } \mathcal C_{x}^{\lattice,h}  \text{ the connected component of }  E_{\varphi_\lattice}^{\geq h}  \text{ containing } x \in \lattice.
%\end{align*}
%Finally, we abbreviate o$\ \coloneqq \pi_N(0) \in \torus$ for all $N\geq 1$.}

Theorem \ref{theorem3.1} below shows that in the subcritical phase $h > h_\star$ with high probability for large $N$ there does not exist a macroscopic connected component of $E_{\Psi_\torus}^{\geq h}$. Actually, there is a large $h$ regime ($h>h_{\star\star}$) where with high probability for large $N$ all connected components of $E_{\Psi_\torus}^{\geq h}$ are  microscopic. This is Theorem \ref{theorem3.2}. In the supercritical phase $h<h_\star$ Theorem \ref{theorem3.4}   establishes with high probability for large $N$ the existence of a macroscopic (in diameter sense) connected component of $E_{\Psi_\torus}^{\geq h}$. The proofs of these results share common features with their counterparts in \cite{5} resp.~\cite{6} in the case of the vacant set of simple random walk on $\torus$ run up to time $uN^d$ and random interlacements at level $u$ on $\lattice$.
The section ends with open questions in the supercritical regime $h<h_\star$, see Remark \ref{remark3.5}.

\begin{theorem}[subcritical phase, $h>h_\star$] \label{theorem3.1}
For all $h> h_\star$ and $\xi>0$ it holds that
\begin{equation*}
\lim_{N\to\infty} \mathbb P^\torus \Big[ \big| \mathcal C_{\textup{max}}^{\torus,h} \big| \geq \xi N^d \Big] = 0.
\end{equation*}
\end{theorem}

\begin{proof}
Consider $h> h_\star$ and $\xi>0$. Choose $\varepsilon>0$  such that $h-\varepsilon>h_\star$. By Markov's inequality it holds that
\begin{equation} \label{3.1}
\begin{split}
\mathbb P^\torus \Big[ \big| \mathcal C_{\textup{max}}^{\torus,h} \big| &\geq \xi N^d \Big] = \mathbb P^\torus \bigg[ \sum_{x \in \torus} \mathds{1}_{ \big\{ | \mathcal C_{x}^{\torus,h} | \geq \xi N^d \big \}}   \geq \xi N^d  \bigg]  \\
&\leq \frac{1}{\xi N^d} \mathbb E^\torus \bigg[ \sum_{x \in \torus} \mathds{1}_{ \big\{ | \mathcal C_{x}^{\torus,h} | \geq \xi N^d \big \}}  \bigg] = \frac1{\xi} \mathbb P^\torus \Big[ \big| \mathcal C_{\text{o}}^{\torus,h} \big| \geq \xi N^d \Big].
\end{split}
\end{equation}
Choose any $\delta \in (\frac12,1)$ and set $a_N \coloneqq \xi N^d - | \torus \setminus B_N^\delta |$. Then, Corollary \ref{corollary2.3} implies that
\begin{equation} \label{3.2}
\begin{split}
\mathbb P^\torus &\Big[ \big| \mathcal C_{\text{o}}^{\torus,h} \big| \geq \xi N^d \Big] \leq \mathbb P^\torus \Big[ \big| \mathcal C_{\text{o}}^{\torus,h} \cap B_N^\delta \big| \geq a_N \Big]   \\
&\leq \mathbb Q_N \Big[  \big|   \mathcal C_{0}^{\lattice,h-\varepsilon} \cap \mathcal B_N^\delta \big| \geq a_N    \Big]  +  \mathbb Q_N \Big[ \big( E_{\Psi_\torus}^{\geq h} \cap B_N^\delta \big) \nsubseteq  \pi_N \big( E_{\varphi_\lattice}^{\geq h-\varepsilon} \cap \mathcal B_N^\delta \big)   \Big]  \\
&\leq \mathbb P^\lattice \Big[  \big|   \mathcal C_{0}^{\lattice,h-\varepsilon} \big| \geq a_N    \Big]   +  \mathbb Q_N \Big[ \big( E_{\Psi_\torus}^{\geq h} \cap B_N^\delta \big) \nsubseteq  \pi_N \big( E_{\varphi_\lattice}^{\geq h-\varepsilon} \cap \mathcal B_N^\delta \big)   \Big]  \\
&\xrightarrow[\eqref{2.13}]{N\to \infty} \mathbb P^\lattice \Big[  \big|   \mathcal C_{0}^{\lattice,h-\varepsilon} \big| = \infty    \Big]  \overset{h-\varepsilon>h_\star}{=} 0,
\end{split}
\end{equation}
since for large $N$, $a_N \geq \xi N^d - (N^d- (N-N^\delta-2)^d) = (\xi+(1-N^{\delta-1}-\frac2N)^d-1)N^d \longrightarrow \infty$ as $N$ goes to infinity. The combination of \eqref{3.1} and \eqref{3.2} concludes the proof.
\end{proof}

\begin{theorem}[strongly non-percolative phase, $h>h_{\star\star}$] \label{theorem3.2}
For all $h>h_{\star\star}$, $\lambda > d$ and $\rho>0$ it holds that
\begin{equation*}
\lim_{N\to\infty} N^\rho \mathbb P^\torus \Big[ \big| \mathcal C_{\textup{max}}^{\torus,h} \big| \geq (\ln(N))^\lambda \Big] = 0.
\end{equation*}
\end{theorem}

\begin{proof}
Fix $h>h_{\star\star}$ and choose $\varepsilon>0$  such that $h-\varepsilon > h_{\star\star}$. Consider $\lambda>d$, $\rho>0$ and $\delta \in (\frac12,1)$. For $N$ large a union bound and the couplings $\mathbb Q_N$ from Corollary \ref{corollary2.3} lead to
\begin{equation}  \label{3.3}
\begin{split}
N^\rho \mathbb P^\torus \Big[ &\big| \mathcal C_{\textup{max}}^{\torus,h} \big| \geq (\ln(N))^\lambda \Big] \leq N^\rho N^d \mathbb P^\torus \Big[ \big| \mathcal C_{\textup{o}}^{\torus,h}  \big| \geq (\ln(N))^\lambda \Big]   \\
&\overset{(*)}{=}  N^{\rho+d} \mathbb P^\torus \Big[ \big| \mathcal C_{\text{o}}^{\torus,h} \cap B_N^\delta \big| \geq  (\ln(N))^\lambda \Big]  \\
&\leq  N^{\rho+d} \mathbb Q_N \Big[  \big|   \mathcal C_{0}^{\lattice,h-\varepsilon} \cap \mathcal B_N^\delta \big| \geq (\ln(N))^\lambda    \Big]  \\ 
&\qquad \qquad \qquad \qquad +  N^{\rho+d} \mathbb Q_N \Big[ \big( E_{\Psi_\torus}^{\geq h} \cap B_N^\delta \big) \nsubseteq  \pi_N \big( E_{\varphi_\lattice}^{\geq h-\varepsilon} \cap \mathcal B_N^\delta \big)   \Big],
\end{split}
\end{equation}
where in $(*)$ we use that $\textup{o} \in \torus$ is at distance larger or equal $\frac{N-N^\delta-2}{2} \geq (\ln(N))^\lambda$ from the boundary of the box $B_N^\delta$. The second term on the right hand side of \eqref{3.3} converges to zero by \eqref{2.12}. It remains to control the first term. By a union bound and \cite{13}, Theorem 2.1, one obtains
\begin{equation*}
\begin{split}
N^{\rho+d} \mathbb Q_N \Big[ & \big|   \mathcal C_{0}^{\lattice,h-\varepsilon} \cap \mathcal B_N^\delta \big| \geq (\ln(N))^\lambda  \Big] \\
&\leq N^{\rho+d} \mathbb P^\lattice \Big[  0 \xleftrightarrow{\varphi_\lattice \geq h-\varepsilon} x \text{ for some $x\in \lattice$ with } d_\lattice(0,x) = \frac12\lfloor (\ln(N))^{\frac{\lambda}d}  \rfloor       \Big]  \\
&\leq \begin{cases} N^{\rho+d} (\ln(N))^{\frac{\lambda}d (d-1)} c_h \exp \big(-c'_h (\ln(N))^{\lambda / d} \big),  &\text{if } d\geq 4   \\
				N^{\rho+d} (\ln(N))^{\frac{\lambda}d (d-1)} c_h \exp \Big(-c'_h \frac{(\ln(N))^{\lambda / d}}{(\ln(\ln(N))^{\lambda / d
			})^6} \Big) ,  &\text{if } d=3      \end{cases}   \\
&\xrightarrow[\lambda>d]{N\to \infty} 0
\end{split}
\end{equation*}
and the proof is complete.
\end{proof}

In the next theorem, $\text{diam}\big(\mathcal C_{\text{o}}^{\torus,h} \big)$ stands for the extrinsic diameter of the connected component of $E_{\Psi_\torus}^{\geq h}$ containing o$\ =\pi_N(0) \in \torus$, i.e.
\begin{equation*}
\text{diam}\big(\mathcal C_{\text{o}}^{\torus,h}\big) \coloneqq \max \Big\{ d_\torus(x,y) \, \Big| \, x,y \in \mathcal C_{\text{o}}^{\torus,h} \Big\} \leq N.
\end{equation*}

\begin{theorem}[supercritical phase, $h<h_\star$]   \label{theorem3.4}
For all $h<h_\star$ it holds that
\begin{equation}  \label{3.4}
\lim_{N\to\infty} \mathbb P^\torus \Big[ \textup{diam}\Big(\mathcal C_{\textup{o}}^{\torus,h}\Big) \geq \frac{N}4  \Big] = \mathbb P^\lattice \big[ 0 \xleftrightarrow{\varphi_\lattice \geq h} \infty \big] > 0
\end{equation}
(the limit actually holds for $h>h_\star$, too; in that case it  equals  zero).
\end{theorem}

\begin{remark} \label{remark3.3}
In the supercritical phase $h<h_\star$ we $\mathbb P^\lattice$-almost surely have  uniqueness of the infinite connected component contained in $E_{\varphi_\lattice}^{\geq h}$. Moreover, the percolation probability $\eta(h) \coloneqq \mathbb P^\lattice [ 0 \xleftrightarrow{\varphi_\lattice \geq h} \infty ]$ is left-continuous on $\mathbb R$ and continuous on $(-\infty,h_\star)$.
The uniqueness property of the infinite connected component follows from the Burton-Keane Theorem (see \cite{8}, Theorem 12.2). The continuity properties of $\eta(\cdot)$ can be derived along the same lines as in the classical setting of Bernoulli bond percolation (see \cite{9}, Lemma 8.9 and Lemma 8.10). The idea goes back to \cite{11}. For the reader's convenience we include a proof in the appendix, see Lemma \ref{lemmaA.1}. \qed
\end{remark}

\begin{proof}[Proof of Theorem \ref{theorem3.4}]
As above, we let  $\eta(h) = \mathbb P^\lattice [ 0 \xleftrightarrow{\varphi_\lattice \geq h} \infty ]$ denote the percolation probability. Assume $h \in \mathbb R$. We claim that for all $\varepsilon >0$ it holds that
\begin{equation} \label{3.5}
\begin{split}
\eta(h+\varepsilon) &\leq  \varliminf_{N \to \infty} \mathbb P^\torus \Big[ \textup{diam}\Big(\mathcal C_{\textup{o}}^{\torus,h}\Big) \geq \frac{N}4  \Big] \leq  \varlimsup_{N \to \infty} \mathbb P^\torus \Big[ \textup{diam}\Big(\mathcal C_{\textup{o}}^{\torus,h}\Big) \geq \frac{N}4  \Big] \leq  \eta(h-\varepsilon). 
\end{split}
\end{equation}
If one sends $\varepsilon$ to zero in \eqref{3.5}, by the continuity of $\eta(\cdot)$ on $\mathbb R\setminus \{h_\star \}$ (see Remark \ref{remark3.3} and Lemma \ref{lemmaA.1}) we recover the desired statement \eqref{3.4} (if $h>h_\star$, then $\eta(h-\varepsilon)=0$ for $\varepsilon>0$ small). Thus, it remains to prove the claim \eqref{3.5}.

Pick some $\delta \in (\frac12,1)$. Note that for large $N$ the boxes $\mathcal B_N^\delta$ and $B_N^\delta$ from \eqref{extraequation14} and \eqref{extraequation12} satisfy 
\begin{align}
\Big\{ \textup{diam}\Big(\mathcal C_0^{\lattice,h}\Big) \geq \frac{N}4 \Big\}  &=  \Big\{ \textup{diam}\Big(\mathcal C_0^{\lattice,h} \cap \mathcal B_N^\delta \Big) \geq \frac{N}4 \Big\}   \label{3.7}   \\
\Big\{ \textup{diam}\Big(\mathcal C_{\textup{o}}^{\torus,h}\Big) \geq \frac{N}4 \Big\}  &=  \Big\{ \textup{diam}\Big(\mathcal C_{\textup{o}}^{\torus,h} \cap B_N^\delta \Big) \geq \frac{N}4 \Big\}.   \label{3.6} 
\end{align} 
Consider $\varepsilon>0$. For $N$ large enough one has by Corollary \ref{corollary2.3}
\begin{align*}
\mathbb P^\torus \Big[ & \textup{diam}\Big(\mathcal C_{\textup{o}}^{\torus,h}\Big) \geq \frac{N}4  \Big]  \\
&\overset{\eqref{3.6}}{\leq} \mathbb Q_N \Big[   \textup{diam}\Big(\mathcal C_0^{\lattice,h-\varepsilon} \cap \mathcal B_N^\delta \Big) \geq \frac{N}4  \Big] +  \mathbb Q_N \Big[ \big( E_{\Psi_\torus}^{\geq h} \cap B_N^\delta \big) \nsubseteq  \pi_N \big( E_{\varphi_\lattice}^{\geq h-\varepsilon} \cap \mathcal B_N^\delta \big)   \Big]  \\
&\overset{\phantom{\eqref{3.7}}}{\leq} \mathbb P^\lattice \Big[   \textup{diam}\Big(\mathcal C_0^{\lattice,h-\varepsilon}  \Big) \geq \frac{N}4  \Big] +  \mathbb Q_N \Big[ \big( E_{\Psi_\torus}^{\geq h} \cap B_N^\delta \big) \nsubseteq  \pi_N \big( E_{\varphi_\lattice}^{\geq h-\varepsilon} \cap \mathcal B_N^\delta \big)   \Big]  \\
&\xrightarrow[\eqref{2.13}]{N\to \infty} \eta(h-\varepsilon)
\end{align*}
and similarly 
\begin{align*}
\mathbb P^\torus \Big[ & \textup{diam}\Big(\mathcal C_{\textup{o}}^{\torus,h}\Big) \geq \frac{N}4  \Big]  \\
&\overset{\phantom{\eqref{3.6}}}{\geq} \mathbb Q_N \Big[   \textup{diam}\Big(\mathcal C_0^{\lattice,h+\varepsilon} \cap \mathcal B_N^\delta \Big) \geq \frac{N}4  \Big] -  \mathbb Q_N \Big[ \pi_N \big( E_{\varphi_\lattice}^{\geq h+\varepsilon} \cap \mathcal B_N^\delta \big) \nsubseteq \big( E_{\Psi_\torus}^{\geq h} \cap B_N^\delta \big)   \Big]  \\
&\overset{\eqref{3.7}}{=} \mathbb P^\lattice \Big[   \textup{diam}\Big(\mathcal C_0^{\lattice,h+\varepsilon}  \Big) \geq \frac{N}4  \Big] -  \mathbb Q_N \Big[ \pi_N \big( E_{\varphi_\lattice}^{\geq h+\varepsilon} \cap \mathcal B_N^\delta \big) \nsubseteq \big( E_{\Psi_\torus}^{\geq h} \cap B_N^\delta \big)   \Big]    \\
&\xrightarrow[\eqref{2.13}]{N\to \infty} \eta(h+\varepsilon).
\end{align*}
This shows the claim \eqref{3.5} and concludes the proof of Theorem \ref{theorem3.4}.
\end{proof}       

\begin{remark} \label{remark3.5}
The above Theorem \ref{theorem3.4} states, in the supercritical phase $h<h_\star$, the existence with high probability of a macroscopic connected component of $E_{\Psi_\torus}^{\geq h}$ in diameter sense. It would be desirable to prove the existence in volume sense, i.e.~to show  when $h<h_\star$ that for some $\xi \in (0,1)$ it holds that
\begin{equation*}
\lim_{N\to\infty} \mathbb P^\torus \Big[ \big| \mathcal C_{\textup{max}}^{\torus,h} \big| \geq \xi N^d \Big] = 1
\end{equation*}
as counterpart to Theorem \ref{theorem3.1} for the subcritical phase $h>h_\star$. There would still be the question of the uniqueness of the giant connected component for $h<h_\star$. Such uniqueness can plausibly be obtained by going below another critical parameter $\overline{h} \leq h_\star$ (see \cite{22}, equation (5.3)) characterising a strong percolative regime of $E_{\varphi_\lattice}^{\geq h}$ (see \cite{22}, equations (5.1) and (5.2)). The equality $\overline{h}=h_\star$ is presently open. \qed
\end{remark}

\begin{acknowledgement}
The author wishes to thank A.-S.~Sznitman for suggesting the problem and for  the useful discussions  and to express his gratitude for the valuable comments made at various stages of the project. 
\end{acknowledgement}

%APPENDIX
\appendix
\section{Appendix}

For the reader's convenience we provide a proof of the continuity properties of the level-set percolation probability $\mathbb P^\lattice [ 0 \xleftrightarrow{\varphi_\lattice \geq h} \infty ]$ in $h$ along the argument of \cite{9}, Lemma 8.9 and Lemma 8.10, in the Bernoulli bond percolation setting. The general idea is due to \cite{11}.

\begin{lemma} \label{lemmaA.1}
The level-set percolation probability $\eta(h) = \mathbb P^\lattice [ 0 \xleftrightarrow{\varphi_\lattice \geq h} \infty ]$ of the Gaussian free field $\varphi_\lattice$ on $\lattice$ is left-continuous on $\mathbb R$ and continuous on $(-\infty,h_\star)$.
\end{lemma}

\begin{proof}
For $n\geq 1$ we define $B(n) \coloneqq \{-n,\ldots,n \}^d$. Then 
\begin{equation}  \label{A.1}
\eta(h) = \mathbb P^\lattice \Big[ \bigcap_{n\geq 1} \big\{  0 \xleftrightarrow{\varphi_\lattice \geq h} \partial_\lattice B(n) \big \}  \Big]  = \lim_{n\to \infty} \mathbb P^\lattice \Big[0 \xleftrightarrow{\varphi_\lattice \geq h} \partial_\lattice B(n)   \Big]. 
\end{equation}
Since $(\varphi_\lattice(x))_{x\in B(n) \cup \partial_\lattice B(n)}$ has a density (see \eqref{0.2}), the expression $\mathbb P^\lattice \big[0 \xleftrightarrow{\varphi_\lattice \geq h} \partial_\lattice B(n)   \big]$ is a continuous function of $h$. Therefore by \eqref{A.1}, $\eta(\cdot)$ is a decreasing limit of continuous functions and thus upper semicontinuous. As $\eta(\cdot)$ is a non-increasing function, it is thus left-continuous.

To show the right-continuity on $(-\infty,h_\star)$ fix $h<h_\star$ and assume $(h_k)_{k\geq0}$ is a sequence satisfying $h_k \downarrow h$ and $h_k < h_\star$ for all $k\geq 0$. As $\eta(h) \geq \eta(h_k)$ for all $k\geq 0$, we have 
\begin{equation} \label{A.2}
0 \leq \eta(h) - \lim_{k\to \infty} \eta(h_k) = \mathbb P^\lattice \Big[ 0 \xleftrightarrow{\varphi_\lattice \geq h} \infty  \text{ but } 0 \centernot{\xleftrightarrow{\varphi_\lattice \geq h_k}} \infty \text{ for all } k\geq 0 \Big] .
\end{equation}
Assume by contradiction  that the probability in \eqref{A.2} is not equal to zero. Consider any $k_0\geq 0$. Since $h<h_{k_0}<h_\star$, we have $E_{\varphi_\lattice}^{\geq h_{k_0}} \subseteq E_{\varphi_\lattice}^{\geq h}$ and by Remark \ref{remark3.3} the two level  sets $\mathbb P^\lattice$-almost surely contain  a unique infinite connected component. Therefore $\mathbb P^\lattice$-almost surely, the unique infinite connected component of $E_{\varphi_\lattice}^{\geq h_{k_0}}$, call it $\mathcal C_{\text{inf}}^{h_{k_0}}$, is contained in the unique infinite connected component of $E_{\varphi_\lattice}^{\geq h}$, which on the event $\{ 0 \xleftrightarrow{\varphi_\lattice \geq h} \infty \}$  necessarily  coincides with  $\mathcal C_0^{\lattice,h}$. 
Hence, under our assumption, we can find a (deterministic) sequence of nearest-neighbour vertices $(x_i)_{0\leq i \leq n}$ in $\lattice$ with $x_0=0$  such that
\begin{equation*}
P^\lattice \Big[ \varphi_\lattice(x_i) \geq h \text{ for } i=0,\ldots,n \text{ and } x_n \in \mathcal C_{\text{inf}}^{h_{k_0}}  \text{ and }  0 \centernot{\xleftrightarrow{\varphi_\lattice \geq h_k}} \infty \text{ for all } k\geq 0  \Big] > 0.
\end{equation*}
Since $\mathbb P^\lattice$-almost surely $\min_{i=0,\ldots,n} \varphi_\lattice(x_i) \neq h$ and $h_k \downarrow h$, it follows that for some $k_1 \geq 0$ with $h_{k_1}\leq h_{k_0}$ we have
\begin{equation*}
P^\lattice \Big[ \varphi_\lattice(x_i) \geq h_{k_1} \text{ for } i=0,\ldots,n \text{ and } x_n \in \mathcal C_{\text{inf}}^{h_{k_0}}  \text{ and }  0 \centernot{\xleftrightarrow{\varphi_\lattice \geq h_k}} \infty \text{ for all } k\geq 0  \Big] > 0,
\end{equation*}
which is a contradiction (since $\mathcal C_{\text{inf}}^{h_{k_0}} \subseteq \mathcal C_{\text{inf}}^{h_{k_1}}$). Thus, the last term of \eqref{A.2} vanishes and the proof of the right-continuity of $\eta(\cdot)$ on $(-\infty,h_\star)$ is complete.
\end{proof}

%BIBLIOGRAPHY

\bibliographystyle{abbrv}
\bibliography{bibliographyfile}

\end{document}